\documentclass[11pt]{amsart}

\usepackage[T1]{fontenc}
\usepackage{lmodern}
\usepackage{geometry}
\usepackage{amsmath}
\usepackage{amssymb}
\usepackage{amsfonts}
\usepackage{mathrsfs}
\usepackage{enumitem}
\usepackage{graphicx}

\theoremstyle{plain}
\newtheorem{theorem}{Theorem}[section]
\newtheorem{lemma}[theorem]{Lemma}

\theoremstyle{definition}
\newtheorem{definition}[theorem]{Definition}

\theoremstyle{remark}

\title[The Daugavet property for Sobolev spaces over the plane]
{The Daugavet property for Sobolev spaces over the plane}

\author{Samir Hamad}

\email{samih49@zedat.fu-berlin.de}

\date{January 2026}


\begin{document}

\begin{abstract}
We show that $W^{1,1}(\mathbb{R}^2)$ has the Daugavet property when endowed with the norm induced by the $L^1$-norm of the gradient, but fails to have the slice diameter two property when equipped with the usual Sobolev norm.
\end{abstract}

\maketitle

\section{Introduction}
We say that a normed space $X$ has the \emph{Daugavet property}, shortly $X \in \mathrm{DPr}$, if
\[
\|\mathrm{Id} + T\| = 1 + \|T\|,
\]
for every rank-one operator $T \colon X \to X$. This property was introduced in \cite{kadets2000banach}, where the following geometric characterization was observed.

\begin{lemma}
The following assertions are equivalent:
\begin{enumerate}
\item[(i)] $X$ has the Daugavet property.
\item[(ii)] For every $y \in S_X$, $x^{*} \in S_{X^{*}}$, and every $\varepsilon > 0$, there exists $x \in S_X$ such that
\[
x^{*}(x) \ge 1 - \varepsilon
\quad \text{and} \quad
\|x + y\| \ge 2 - \varepsilon.
\]
\item[(iii)] For every $\varepsilon > 0$ and every $y \in S_X$, the closed convex hull of the set
\[
\{\, u \in (1+\varepsilon) B_X : \|y + u\| \ge 2 - \varepsilon \,\}
\]
contains $S_X$.
\end{enumerate}
\end{lemma}

I.\,K.~Daugavet was the first to show that $C[0,1]$ enjoys the Daugavet property \cite{daugavet1963property}, and a few years later it was shown that $L^1[0,1]$ also has the Daugavet property \cite{lozanovskii1966almost}. From the geometric characterization, it follows easily that every slice of $B_X$, i.e., sets of the form $\operatorname{Slice}(B_X,x^*,\varepsilon) := \{ x \in B_X \mid x^*(x) > 1 - \varepsilon\}$ for $x^* \in S_{X^*}$ and $\varepsilon > 0$, has diameter $2$, i.e., $X$ has the \emph{slice diameter two property}. Hence a Banach space with the Daugavet property cannot possess the Radon--Nikod\'ym property. A way deeper isomorphic consequence is that no Banach space with the Daugavet property embeds into a space with an unconditional basis; see \cite[Corollary 16]{shvydkoy2000geometric}.
  The following notion of Daugavet points was introduced in \cite{abrahamsen2020delta}. \begin{definition} Let $X$ be a normed space. An element $y \in S_X$ is called a \emph{Daugavet point} if, for every $\varepsilon > 0$, $B_X = \overline{\operatorname{conv}}\bigl\{\, u \in (1+\varepsilon) B_X : \|y + u\| \ge 2 - \varepsilon \,\bigr\}$. \end{definition}

It is easy to see that a normed space $X$ enjoys the Daugavet property whenever the set of Daugavet points is dense in $S_X$. 

Let $M$ be a pointed metric space. The space of all Lipschitz maps $F\colon M \to \mathbb{R}$ that vanish at $0$, equipped with the norm
\[
\|F\|_{L} = \sup \left\{ \frac{|F(m_1) - F(m_2)|}{d(m_1,m_2)} : m_1 \neq m_2 \in M \right\},
\]
is called $\operatorname{Lip}_0(M)$. In \cite[Theorem~3.1]{ivakhno2007daugavet} it was proved that $\operatorname{Lip}_0(M)$ has the Daugavet property whenever $M$ is a metric length space, i.e., for every pair of points $x,y \in M$ the distance $d(x,y)$ equals the infimum of the lengths of rectifiable curves joining them. It was later shown in \cite[Proposition~2.4 and Theorem~3.3]{garcia2018characterisation} that these metric spaces are precisely those for which $\operatorname{Lip_0}(M)$ has the Daugavet property. 
 It is therefore natural to investigate the Daugavet property of its $L^1$-counterpart $W^{1,1}(U)$, where $U$ is an open subset of $\mathbb{R}^n$ equipped with the Euclidean norm.
Here, $W^{1,1}(U)$ denotes the set of all functions $f \in L^1(U)$ for which there exists a vector field $\nabla f \in L^1(U;\mathbb{R}^n)$ satisfying
\[
\int_U f \, \operatorname{div} \varphi \, dx = - \int_U \nabla f \cdot \varphi \, dx
\]
for every test function $\varphi \in C_c^1(U;\mathbb{R}^n)$. This space is endowed either with the seminorm
\[
\mathcal{k} f \mathcal{k}_{\tilde{W}^{1,1}} := \mathcal{k} \nabla f \mathcal{k}_{L^1(U;\mathbb{R}^n)},
\]
which becomes a norm on the quotient obtained by factoring out all functions in $W^{1,1}(U)$ with vanishing gradient, or with the usual norm
\[
\mathcal{k}f \mathcal{k}_{W^{1,1}(U)} := \mathcal{k}f \mathcal{k}_{L^1(U)} + \mathcal{k}\nabla f \mathcal{k}_{L^1(U;\mathbb{R}^n)}.
\]

To distinguish \((W^{1,1}(U), \|\cdot\|_{\tilde{W}^{1,1}})\) from the Sobolev space equipped with its usual norm, we shall denote it by \(\tilde{W}^{1,1}(U)\). The space \(\tilde{W}^{1,1}(U)\) is also known as \(\dot{W}^{1,1}(U) \cap L^1(U)\), where \(\dot{W}^{1,1}(U)\) denotes the homogeneous Sobolev space.

\section{Preliminaries}

Let $g \in S_{L^1(U;\mathbb{R}^n)}$ and fix $\varepsilon > 0$. $\mathcal{L}^n$ denotes the $n$-dimensional Lebesgue measure. Then there exists $\delta > 0$ such that for every measurable set $V \subset U$ with $\mathcal{L}^n(V) < \delta$ we have
\[
\int_V |g| \, dx < \varepsilon.
\]
Consequently, for any $f \in S_{L^1(U;\mathbb{R}^n)}$ whose essential support has Lebesgue measure smaller than $\delta$, we obtain $\|f+g\|_{L^1} > 2 - 2\varepsilon$. Therefore, to show that \(\tilde{W}^{1,1}(U)\) is a Daugavet space, it suffices to prove that for every
\(f \in B_{\tilde{W}^{1,1}(U)}\) the gradient \(\nabla f\) can be approximated arbitrarily well by finite
convex combinations of functions \((f_i)\) in \(S_{\tilde{W}^{1,1}(U)}\) such that each \(f_i\) is
nonconstant only on a set of arbitrarily small measure. To this end, we first examine the level sets of a smooth function with compact support.

\begin{definition}
Let $\Phi : M \to N$ be a smooth map between smooth manifolds. A point $p \in M$ is called a \emph{regular point} of $\Phi$ if the differential
\[
d\Phi_p : T_p M \to T_{\Phi(p)} N
\]
is surjective; otherwise $p$ is called a \emph{critical point} of $\Phi$. A point $c \in N$ is called a \emph{regular value} of $\Phi$ if every point of the level set $\Phi^{-1}(c)$ is a regular point of $\Phi$; otherwise $c$ is called a \emph{critical value}. 
\end{definition}

For a smooth map $f \colon \mathbb{R}^n \to \mathbb{R}$, a point $c \in \mathbb{R}$ is a regular value if $\nabla f(x) \neq 0$ for every $x \in f^{-1}(c)$. The following theorem, due to Sard, states that the set of regular values has full measure; see \cite[p.~129]{lee2003smooth}.

\begin{theorem}[Sard's Theorem]
Suppose that $M$ and $N$ are smooth manifolds and let $\Phi : M \to N$ be a smooth map.
Then the set of critical values of $F$ has measure zero in $N$.
\end{theorem}

Regular values play a central role in the sequel due to the following preimage theorem; see \cite[p.~106]{lee2003smooth}.

\begin{theorem}[Preimage Theorem]
If $c \in N$ is a regular value, then for a smooth map
$\Phi : M \rightarrow N$ the preimage $\Phi^{-1}(c)$ is a properly embedded
submanifold of dimension $\dim M - \dim N$.
\end{theorem}

Note that every smooth connected $1$-manifold is diffeomorphic either to $S^1$ or to $\mathbb{R}$; see \cite[p.~398]{lee2003smooth}, and that the connected components of a smooth manifold are open in the subspace topology and are themselves smooth manifolds of the same dimension. By decomposing $f^{-1}(c)$ into its disjoint connected components for a fixed $f \in C_c^\infty(\mathbb{R}^2;\mathbb{R})$ and a regular value $c$, we then conclude from the preceding theorem that the level set $f^{-1}(c)$ is a finite disjoint union of Jordan curves, each diffeomorphic to $S^1$, since $\mathbb{R}$ cannot occur among the connected components of $f^{-1}(c)$ due to the compactness of $f^{-1}(c)$. Using Sard's theorem, we conclude that for almost every $c \in \mathbb{R}$, the level set $f^{-1}(c)$ is a disjoint union of finitely many Jordan curves, each diffeomorphic to $S^1$.

\setlength{\parskip}{1em}
By the Jordan curve theorem, every Jordan curve $J$ separates the plane into a bounded region $\operatorname{int} J$ and an unbounded region $\operatorname{ext} J$; by a theorem of Schoenflies, the interior $\operatorname{int} J$ is homeomorphic to the open unit ball in $\mathbb{R}^2$, and hence homeomorphic to $\mathbb{R}^2$; see \cite[Theorem~3.1]{thomassen1992jordan}. Using a theorem of Janiszewski together with a straightforward induction argument, we obtain the following lemma, which is also the reason why we work in the plane; see \cite[Theorem~1]{bing1946generalizations}.

\begin{lemma}
Let $J_1$ be a Jordan curve in $\mathbb{R}^2$, and let $\{J_i\}_{i=2}^m$ be a finite collection of pairwise disjoint Jordan curves contained in $\operatorname{int} J_1$. Then \[
\operatorname{int} J_1 \setminus \bigcup_{i=2}^m \overline{\operatorname{int} J_i},
\]
is path-connected. Moreover, for any finite collection of pairwise disjoint Jordan curves $\{J_i\}_{i=1}^n$ in $\mathbb{R}^2$, the complement $\mathbb{R}^2 \setminus \bigcup_{i=1}^n \overline{\operatorname{int} J_i}$ is path-connected.
\end{lemma}

We will also make use of the coarea formula; for a proof, see \cite[p.~134--139]{evans2018measure}.

\begin{theorem}[Coarea formula]
Let $f \colon \mathbb{R}^n \to \mathbb{R}$ be a Lipschitz continuous function with $n \ge 1$. Then, for every $\mathcal{L}^n$-measurable set $A \subset \mathbb{R}^n$, the following identity holds:
\[
\int_A |\nabla f(x)|\,\mathrm{d}x
=
\int_{\mathbb{R}} 
\mathcal{H}^{n-1}\!\left(A \cap f^{-1}(\{t\})\right)\,\mathrm{d}t,
\]
where $\mathcal{H}^{n-1}$ denotes the $(n-1)$-dimensional Hausdorff measure.
\end{theorem}

Note that, as a corollary, we have the known formula $\int_{\mathbb{R}^n} |\nabla f| = \int_\mathbb{R} \mathcal{H}^{1}(f^{-1}(t)) \, dt$, and therefore $\mathcal{H}^{1}(f^{-1}(\cdot)) \in L^1(\mathbb{R})$ for $f \in \tilde{W}^{1,1}(\mathbb{R}^2)$.

\section{Main results}

\begin{theorem}
The normed space $\tilde{W}^{1,1}(\mathbb{R}^2)$ enjoys the Daugavet property.   
\end{theorem}

\begin{proof}
Let $f \in C_c^\infty(\mathbb{R}^2)$ and $\varepsilon > 0$ with $\|f\|_{\tilde{W}^{1,1}} = 1$, and denote $f(\mathbb{R}^2) = [a,b]$. By the previous discussion, for every regular value $c \in \mathbb{R}$ the level set $f^{-1}(c)$ is a finite disjoint union of Jordan curves, each diffeomorphic to $S^1$. In particular, $\mathcal{L}^2(f^{-1}(c)) = 0$. Let $N \subset [a,b]$ denote the set of non-regular values, which has Lebesgue measure zero by Sard's theorem. Furthermore, let $\mathcal{F}$ be the collection of all intervals of the form $(x - p_x, x + p_x) \subset (a,b)$ such that $x$, $x - p_x$, and $x + p_x$ are regular values and $\mathcal{L}^2(f^{-1}((x - p_x, x + p_x))) < \varepsilon$.
 Then $\mathcal{F}$ forms a Vitali covering of $(a,b) \setminus N$, and hence, by Vitali's covering theorem; see \cite[p.~319]{elstrodt1996mass}, there exists a countable disjoint subcollection $((q_i - p_i, q_i + p_i))_{i=1}^\infty \subset \mathcal{F}$ such that $
\mathcal{L}\Big([a,b] \setminus( N\,\cup\, \bigcup_{i=1}^\infty (q_i - p_i, q_i + p_i))\Big) = 0$ and therefore 
\[
\mathcal{L}\Big([a,b] \setminus\bigcup_{i=1}^\infty (q_i - p_i, q_i + p_i)\Big)\Big) = 0 \tag{1}
\]
Fix $i \in \mathbb{N}$ and consider the open interval $(q_i-p_i, q_i+p_i)$. Let 
\[
f^{-1}((q_i-p_i, q_i+p_i)) = \bigcup_{j=1}^\infty A_j
\] 
be the decomposition into the disjoint open connected components. Then, since $f \in C_c^\infty(\mathbb{R}^2) \subset C(\mathbb{R})$, the boundary of each component satisfies
\[
\partial A_j \subset f^{-1}(\{q_i-p_i\}) \cup f^{-1}(\{q_i+p_i\}) = \bigcup_{k=1}^m C_k,
\] 
where the $C_k$ are the disjoint Jordan curves corresponding to the level sets $f^{-1}(\{q_i-p_i\})$ and $f^{-1}(\{q_i+p_i\})$, which exist since, by construction, $\{q_i-p_i\}$ and $\{q_i+p_i\}$ are regular.
 
 Now fix $d \in \mathbb{N}$. Consider all $C_k$ satisfying $\operatorname{int} C_k \cap A_d \neq \emptyset$. This collection of Jordan curves is nonempty, since every $a \in A_d$ is enclosed by some $C_k$. Otherwise, Lemma~2.4 and $A_d$ being the connected component would imply that $A_d = \mathbb{R}^2 \setminus \bigcup_{k = 1}^m \overline{\operatorname{int}C_k}$, which contradicts $\mathcal{L}^2(A_d) \leq \mathcal{L}^2(f^{-1}((q_i-p_i, q_i+p_i))) < \varepsilon < \infty$. Now choose $C_w$ with $\operatorname{int} C_w \cap A_d \neq \emptyset$ such that for every $C_k \subset \operatorname{int} C_w$ we have $\operatorname{int} C_k \cap A_d = \emptyset$. This is possible since the collection $\{C_k\}$ is finite. Since $A_d$ is connected, it follows that $A_d \subset \operatorname{int} C_w$.

Suppose that $x \in (\operatorname{int} C_w) \setminus A_d$. Also suppose that $x$ is not contained in $\overline{\operatorname{int} C_k}$ for any $C_k \subset \operatorname{int} C_w$. By the condition on $C_w$, we also see that no point of $\operatorname{int}(C_w) \cap A_d$ is enclosed by a Jordan curve $C_k$ contained in $\operatorname{int} C_w$. By Lemma~2.4, we can therefore construct a path from $x$ to an arbitrary point in $\operatorname{int}(C_w) \cap A_d$ that does not intersect any $C_k \subset \operatorname{int} C_w$.
 Since $A_d$ is path-connected, this implies that $x \in A_d$, which yields a contradiction. Therefore,
\[
\operatorname{int} C_w
= A_d \,\cup\, \bigcup_{C_j \subset \operatorname{int} C_w} \overline{\operatorname{int} C_j}. \tag{2}
\]

By~(2), there are only finitely many connected components $A_j$, since there are only finitely many Jordan curves $C_k$. Without loss of generality suppose that $C_w$ belongs to the level set of $p_i+q_i$. The Jordan curves $C_k$ inside $\operatorname{int}C_w$ such that $\operatorname{int}C_k$ is maximal by inclusion are denoted by $(\tilde{C}_l)_{l =1}^{\tilde{m}}$. We may therefore define a function in $\tilde{W}^{1,1}(\mathbb{R}^2)$ by
\[
f_i^d(x) :=
\begin{cases}
0, & \text{if } x \in \operatorname{ext} C_w, \\[6pt]
f(x) - (p_i + q_i), & \text{if } x \in A_d, \\[6pt]
0, & \text{if } x \in \overline{\operatorname{int} \tilde{C}_l} \subset \operatorname{int} C_w
      \text{ and } \tilde{C}_l \text{ belongs to the level set } p_i + q_i, \\[6pt]
-2 q_i, & \text{if } x \in \overline{\operatorname{int} \tilde{C}_l} \subset \operatorname{int} C_w
      \text{ and } \tilde{C}_l \text{ belongs to the level set } p_i - q_i.
\end{cases}
\]

Note that $\overline{\operatorname{int} C_k} \subset \operatorname{int} C_w$ is equivalent to $C_k \subset \operatorname{int} C_w$. By construction, the values prescribed on adjacent regions coincide on their common boundaries, yielding a continuous transition between the different cases and therefore $f_i^d \in \tilde{W}^{1,1}(\mathbb{R}^2)$ with $\nabla f_i^d = \nabla f$ on $A_d$ and zero elsewhere. For clearness, we fully write out the standard argument to see this fact. Take $\varphi \in C_c^\infty(\mathbb{R}^2;\mathbb{R}^2)$. By the divergence theorem,
\begin{align*}
\int_{\mathbb{R}^2} \operatorname{div}(\varphi)\, f_i^d \, dx 
&= \int_{A_d} \operatorname{div}(\varphi)\, f_i^d \, dx 
+ \sum_{l=1}^{\tilde{m}} \int_{\operatorname{int} \tilde{C}_l} \operatorname{div}(\varphi)\, f_i^d \, dx \\
&= - \int_{A_d} \varphi \cdot \nabla f \, dx 
+ \int_{\partial A_d} (\varphi \cdot N_{\partial A_d}) \, f_i^d \, d\mathcal{H}^1
+ \sum_{l=1}^{\tilde{m}} \int_{\partial \tilde{C}_l} (\varphi \cdot N_{\partial \tilde{C}_l}) \, f_i^d \, d\mathcal{H}^1 \\
&= - \int_{A_d} \varphi \cdot \nabla f \, dx  -\sum_{l=1}^{\tilde{m}} \int_{\partial \tilde{C}_l} (\varphi \cdot N_{\partial \tilde{C}_l}) \, f_i^d \, d\mathcal{H}^1
+ \sum_{l=1}^{\tilde{m}} \int_{\partial \tilde{C}_l} (\varphi \cdot N_{\partial \tilde{C}_l}) \, f_i^d \, d\mathcal{H}^1 \\
&= - \int_{A_d} \varphi \cdot \nabla f \, dx ,
\end{align*}
where we used that $\partial A_d = C_w \cup\bigcup_{l=1}^{\tilde{m}} \tilde{C}_l$ and that $N_{\partial A_d}(x) = -N_{\tilde{C}_l}(x)$ for $x \in \tilde{C}_l$, while $N_{\partial A_d}$ and $N_{\tilde{C}_l}$ denote the outer normals of $\partial A_d$ and $\tilde{C}_l$.

Now we define $f_i := \sum_{j} f_i^j$, where the sum is finite since there are only finitely many connected components $A_j$. It holds that $\nabla f_i = \nabla f$ on $f^{-1}((p_i - q_i, p_i + q_i))$, while $\nabla f_i \equiv 0$ elsewhere.
Applying Lemma~2.5 with $A:= f^{-1}([a,b] \setminus \bigcup_{i=1}^{\infty} (q_i - p_i, q_i + p_i))$,
we obtain
\[
\int_{f^{-1}\Big([a,b] \setminus \bigcup_{i=1}^{\infty} (q_i - p_i, q_i + p_i)\Big)} |\nabla f| \, dx
= \int_{[a,b] \setminus \bigcup_{i=1}^{\infty} (q_i - p_i, q_i + p_i)} \mathcal{H}^1(f^{-1}(t)) \, dt
= 0.
\]

by (1). Since $\mathcal{H}^1(f^{-1}(\cdot)) \in L^1(\mathbb{R})$, we obtain
\[
\Bigg\| \nabla\Bigg(\sum_{i=1}^{n} f_i\Bigg) - \nabla f \Bigg\|_{L^1} \longrightarrow 0
\]

as \(n \to \infty\), and hence
\[
\Bigg\| \sum_{i=1}^{m} \|f_i\| \frac{f_i}{\|f_i\|} - f \Bigg\|_{\tilde{W}^{1,1}} < \varepsilon
\]
for \(m \ge n_0\) with \(n_0 \in \mathbb{N}\) sufficiently large.  

Since also
\[
\sum_{i=1}^{\infty} \|f_i\|_{\tilde{W}^{1,1}} = \mathcal{k} f \mathcal{k}_{\tilde{W}^{1,1}} = 1,
\] 
we conclude that \(f\) can be approximated arbitrarily well by convex combinations of functions in \(S_{\tilde{W}^{1,1}}\) whose gradient supports have Lebesgue measure smaller than \(\varepsilon\), where \(\varepsilon>0\) was arbitrary. Together with the first remark in the previous section and the density of \(C_c^\infty(\mathbb{R}^2)\) in \(\tilde{W}^{1,1}(\mathbb{R}^2)\), this completes the proof.
\end{proof}

Define
\[
Q(x,r_1,r_2)
:= \big\{ y \in \mathbb{R}^2 \,\big|\, |y_1 - x_1| < r_1/2,\ |y_2 - x_2| < r_2/2 \big\},
\]
for $x = (x_1,x_2)$. We now prove a far more intuitive fact concerning the Sobolev space
$W^{1,1}(\mathbb{R}^2)$ endowed with its usual Sobolev norm.

\begin{theorem}
The Banach space $W^{1,1}(\mathbb{R}^2)$ does not have the slice diameter two property.
\end{theorem}

\begin{proof}
Let $a > 0$. We denote by $C$ the Poincaré constant such that \[
\|u\|_{L^1(Q(0,1+2a,1+2a))} \leq C \, \|\nabla u\|_{L^1(Q(0,1+2a,1+2a))},
\]
for every $u \in W_0^{1,1}(Q(0,1+2a,1+2a))$ and $b := \frac{1}{8(C+1)}$. Define $D := (Q(0,1+2a,1) \cup Q(0,1,1+2a)) \setminus Q(0,1,1)$ and the vector field
\[
g(x) := - \Big(\operatorname{sign}(x_1)\,\chi_{\{|x_1| > |x_2|\}}, \;\operatorname{sign}(x_2)\,\chi_{\{|x_2| > |x_1|\}}\Big)\,\chi_D(x).
\]

Consider the functional \(T \in W^{1,1}(\mathbb{R}^2)^*\) defined by
\[
Tf := \int_{Q(0,1,1)} f(x) \, dx + \int_{\mathbb{R}^2} g(x) \cdot \nabla f(x) \, dx.
\]

Obviously \(\mathcal{k}T\mathcal{k} \le 1\). For each \(n \in \mathbb{N}\), define
\[
f_n(x) :=
\begin{cases}
1 - n \, \operatorname{dist}(x, Q(0,1,1)), & \text{if } \operatorname{dist}(x, Q(0,1,1)) < \frac{1}{n}, \\[0.3em]
0, & \text{otherwise}.
\end{cases}
\]

It is straightforward to verify that \(\mathcal{k}f_n/5\mathcal{k}_{W^{1,1}(\mathbb{R}^2)} \to 1\) and \(Tf_n/5 \to 1\) as \(n \to \infty\), which implies \(\mathcal{k}T\mathcal{k} = 1\). Choose an arbitrary sequence \(g_n \in \operatorname{Slice}\big(B_{W^{1,1}}, T, 1/n\big)\) with
\(\|g_n\|_{W^{1,1}(\mathbb{R}^2)} = 1\).
 We will show that the Lebesgue measure of $B_n := \{ x \in Q(0,1,1) \mid g_n(x) < b/2\}$ converges to zero. Since $C_c^\infty(\mathbb{R}^2)$ is dense in $W^{1,1}(\mathbb{R}^2)$, we may assume that $g_n \in C_c^\infty(\mathbb{R}^2)$, since convergence in $L^1$ implies convergence in measure. As \(\mathcal{k} g_n \mathcal{k}_{W^{1,1}(A)} \to 0\) for every \(A \subset \mathbb{R}^2 \setminus Q(0,1+2a,1+2a)\), we also may assume that each \(g_n\) vanishes on \(\mathbb{R}^2 \setminus Q(0,1+2a,1+2a)\). For future reference, note that \(\mathcal{k}\nabla g_n \mathcal{k}_{L^1(Q(0,1,1))} \to 0\). Indeed, suppose that there is a non-relabeled subsequence such that 
\[
\mathcal{k} \nabla g_n \mathcal{k}_{L^1(Q(0,1,1))} > c > 0 \quad \text{for every } n.
\] 
Then, since $g \equiv 0$  on $Q(0,1,1)$ and $\mathcal{k} g \mathcal{k}_\infty \le 1$, we get
\begin{align*}
T g_n 
&\le \int_{Q(0,1,1)} |g_n| \, dx 
   + \int_{\mathbb{R}^2 \setminus Q(0,1,1)} |\nabla g_n| \, dx \\
&= \mathcal{k} g_n \mathcal{k}_{L^1(\mathbb{R}^2)} 
   + \mathcal{k} \nabla g_n \mathcal{k}_{L^1(\mathbb{R}^2)} 
   - \mathcal{k} \nabla g_n \mathcal{k}_{L^1(Q(0,1,1))} \\
&\le 1 - c.
\end{align*}

which contradicts the fact that $T g_n \rightarrow 1$.

 We have
\begin{align*}
\|g_n\|_{L^1(\mathbb{R}^2)} + \|\nabla g_n\|_{L^1(\mathbb{R}^2)} - \frac{1}{n} &< T g_n \\
&= \int_{Q(0,1,1)} g_n(x) \, dx + \int_{\mathbb{R}^2} g(x) \cdot \nabla g_n(x) \, dx
\end{align*}
and hence
\begin{align*}
\|\nabla g_n\|_{L^1(\mathbb{R}^2)} - \frac{1}{n} 
&\le 
\int_{[-\frac{1}{2}-a,-\frac{1}{2}] \times [-\frac{1}{2}, \frac{1}{2}]} \frac{\partial g_n}{\partial x_1} \, dx
- \int_{[\frac{1}{2},\frac{1}{2}+a] \times [-\frac{1}{2}, \frac{1}{2}]} \frac{\partial g_n}{\partial x_1} \, dx \\
&\quad + \int_{[-\frac{1}{2},\frac{1}{2}] \times [-\frac{1}{2}-a, -\frac{1}{2}]} \frac{\partial g_n}{\partial x_2} \, dx
- \int_{[-\frac{1}{2},\frac{1}{2}] \times [\frac{1}{2}, \frac{1}{2}+a]} \frac{\partial g_n}{\partial x_2} \, dx. 
\end{align*}

From this we immediately obtain that

\begin{align*}
0
&\le 
\Bigg(
\int_{[-\frac{1}{2}-a,-\frac{1}{2}] \times [-\frac{1}{2}, \frac{1}{2}]}
\frac{\partial g_n}{\partial x_1} \, dx
- \frac{1}{4}\|\nabla g_n\|_{L^1(\mathbb{R}^2)}
+ \frac{1}{4n}
\Bigg) \\
&\quad +
\Bigg(
-\int_{[\frac{1}{2},\frac{1}{2}+a] \times [-\frac{1}{2}, \frac{1}{2}]}
\frac{\partial g_n}{\partial x_1} \, dx
- \frac{1}{4}\|\nabla g_n\|_{L^1(\mathbb{R}^2)}
+ \frac{1}{4n}
\Bigg) \\
&\quad +
\Bigg(
\int_{[-\frac{1}{2},\frac{1}{2}] \times [-\frac{1}{2}-a, -\frac{1}{2}]}
\frac{\partial g_n}{\partial x_2} \, dx
- \frac{1}{4}\|\nabla g_n\|_{L^1(\mathbb{R}^2)}
+ \frac{1}{4n}
\Bigg) \\
&\quad +
\Bigg(
-\int_{[-\frac{1}{2},\frac{1}{2}] \times [\frac{1}{2}, \frac{1}{2}+a]}
\frac{\partial g_n}{\partial x_2} \, dx
- \frac{1}{4}\|\nabla g_n\|_{L^1(\mathbb{R}^2)}
+ \frac{1}{4n}
\Bigg) \\
&= I_{n,1} + I_{n,2} + I_{n,3} + I_{n,4},
\end{align*}
where $I_{n,1},\dots,I_{n,4}$ denote the four terms above in their natural order. Therefore, for each $n \in \mathbb{N}$, there exists an index $i \in \{1,2,3,4\}$ such that $I_{n,i} \ge 0$. Hence, there exists a $j \in \{1,2,3,4\}$ such that $I_{n,j} \ge 0$ for infinitely many $n$. Without loss of generality, we may assume that $j = 3$.

Therefore, after taking a subsequence, we may assume that for each $n \in \mathbb{N}$
\[
\int_{[-1/2,1/2] \times [-1/2 - a, -1/2]} \frac{\partial g_n}{\partial x_2} \, dx \geq \frac{\mathcal{k} \nabla g_n \mathcal{k}_{L^1(\mathbb{R}^2 )}}{4} - \frac{1}{4n}.
\]

By the fundamental theorem of calculus, we have
\begin{align*}
\int_{-1/2}^{1/2} g_n(x_1,-1/2) \, dx_1 
&= \int_{-1/2}^{1/2} g_n(x_1,-1/2) \, dx_1 - \int_{-1/2}^{1/2} g_n(x_1,-1/2 - a) \, dx_1 \\
&= \int_{[-1/2,1/2] \times [-1/2 - a, -1/2]} 
    \frac{\partial g_n}{\partial x_2} \, dx \\
&\geq \frac{\mathcal{k} \nabla g_n \mathcal{k}_{L^1(\mathbb{R}^2)}}{4} - \frac{1}{4n}. 
\end{align*}

Applying the fundamental theorem of calculus again, it follows that
\begin{align*}
\int_{-1/2}^{1/2} g_n(x_1, q) \, dx_1 
&\geq \frac{\mathcal{k} \nabla g_n \mathcal{k}_{L^1(\mathbb{R}^2)}}{4} - \frac{1}{4n} - \mathcal{k} \nabla g_n \mathcal{k}_{L^1(Q(0,1,1))}, \tag{1}
\end{align*}
for every $q \in (-1/2, 1/2)$.  We combine $\|\nabla g_n\|_{L^1(\mathbb{R}^2)} = \|\nabla g_n\|_{L^1(Q(0,1+2a,1+2a))} \geq \frac{1}{C+1}$ with (1) to obtain
\[
\int_{-1/2}^{1/2} g_n(x_1, q) \, dx_1 \geq \frac{1}{4(C+1)} - \frac{1}{4n} - \|\nabla g_n\|_{L^1(Q(0,1,1))}.
\]
Choose \(n_0 \in \mathbb{N}\) sufficiently large such that for all \(n \ge n_0\) 
\[
\int_{-1/2}^{1/2} g_n(x_1, q) \, dx_1 \ge \frac{1}{8(C+1)} = b, \tag{2}
\]
for every $q \in (-1/2,1/2)$. Note that we have used the fact that \(\mathcal{k}\nabla g_n \mathcal{k}_{L^1(Q(0,1,1))} \to 0\), which was proved above. After discarding the first $n_0$ terms of the sequence, we may assume that this holds for all $n \in \mathbb{N}$.
Now suppose that 
\[
\mathcal{L}^2(B_n) > d > 0
\]
for a non-relabeled subsequence and a constant $ 1 > d > 0$. Then for the orthogonal projection $P_1$  onto the first axis, we still have $\mathcal{L}(P_1(B_n)) > d$. Note that $B_n$ is open as $g_n$ is smooth. For every $x \in P_1(B_n)$ choose an interval $I_x$ such that
$I_x \times \{h_x\} \subset B_n$ for some $h_x \in (-1/2,1/2)$.
The family $\{I_x\}_{x \in P_1(B_n)}$ forms a Vitali covering of $P_1(B_n)$.
Hence, there exists a finite collection of pairwise disjoint intervals
$(I_{x_i^n})_{i=1}^{m_n}$ with corresponding heights
$(h_{x_i^n})_{i=1}^{m_n}$ such that
$I_{x_i^n} \times \{h_{x_i^n}\} \subset B_n$ for all $i$ and
\[
\mathcal{L}\!\left(P_1(B_n) \setminus \bigcup_{i=1}^{m_n} I_{x_i^n}\right)
<  (\mathcal{L}(P_1(B_n)) - d)/2.
\]
In particular,
\[
\mathcal{L}\!\left(\bigcup_{i=1}^{m_n} I_{x_i^n}\right) > d. \tag{3}
\]
We have
\[
\sum_{i=1}^{m_n} \int_{I_{x_i^n}} g_n(x_1,h_{x_i^n}) \, dx_1
\leq \mathcal{L}( \bigcup_{i=1}^{m_n} I_{x_i^n})\,\frac{b}{2}.
\]
Applying the fundamental theorem of calculus yields
\[
\int_{\bigcup_{i=1}^{m_n} I_{x_i^n}} g_n(x_1,q)\,dx_1
\leq \mathcal{L}\!\left(\bigcup_{i=1}^{m_n} I_{x_i^n}\right)\frac{b}{2}
+ \|\nabla g_n\|_{L^1(Q(0,1,1))},
\]
for every $q \in (-1/2,1/2)$. After integrating, we obtain
\[
\int_{\bigcup_{i=1}^{m_n} I_{x_i^n} \times [-1/2,1/2]} g_n(x_1,x_2)\,dx_1dx_2
\leq \mathcal{L}\!\left(\bigcup_{i=1}^{m_n} I_{x_i^n}\right)\frac{b}{2}
+ \|\nabla g_n\|_{L^1(Q(0,1,1))}.
\]
Consequently, there exists $e_n \in \bigcup_{i=1}^{m_n} I_{x_i^n}$ such that
\[
\mathcal{L}\!\left(\bigcup_{i=1}^{m_n} I_{x_i^n}\right)
\int_{-1/2}^{1/2} g_n(e_n,x_2)\,dx_2
\leq \mathcal{L}\!\left(\bigcup_{i=1}^{m_n} I_{x_i^n}\right)\frac{b}{2}
+ \|\nabla g_n\|_{L^1(Q(0,1,1))},
\]
and hence with (3)
\[
\int_{-1/2}^{1/2} g_n(e_n,x_2)\,dx_2
\leq \frac{b}{2}
+ \frac{1}{d}\,\|\nabla g_n\|_{L^1(Q(0,1,1))}.
\]
Using the fundamental theorem of calculus one last time and integrating with respect to the first variable, we obtain
\[
\int_{-1/2}^{1/2}\int_{-1/2}^{1/2} g_n(x_1,x_2)\,dx_2\,dx_1
\leq \frac{b}{2}
+ \frac{1}{d}\,\|\nabla g_n\|_{L^1(Q(0,1,1))}
+ \|\nabla g_n\|_{L^1(Q(0,1,1))}.
\]
Choosing $n \in \mathbb{N}$ large enough we get 
\[
\int_{Q(0,1,1)}g_n \, dx < 2b/3,
\]
since \(\mathcal{k}\nabla g_n \mathcal{k}_{L^1(Q(0,1,1))} \to 0\). This contradicts~(2). Therefore we have $\mathcal{L}^2(B_n) \rightarrow 0$. 

Suppose now that every slice has diameter two. Then we can choose $f_n,h_n \in \operatorname{Slice}(B_{W^{1,1}},T,1/n)$ such that $\|h_n-f_n\|_{W^{1,1}(\mathbb{R}^2)} \to 2$. We have already shown that the Lebesgue measure of the sets $A_n := \{x \in Q(0,1,1) \mid f_n(x) < b/2\}$ and $C_n :=\{x \in Q(0,1,1)\mid h_n(x) < b/2\}$ converges to zero. Choose $m \in \mathbb{N}$ sufficiently large such that $\mathcal{L}^2(B_m) + \mathcal{L}^2(C_m) < \tfrac14$ and 
\[
\|f_m-h_m\|_{W^{1,1}(\mathbb{R}^2)} > 2-\tfrac{b}{2}. \tag{4}
\]
Then
\[
\mathcal{L}^2\big((Q(0,1,1)\setminus B_m)\cap (Q(0,1,1)\setminus C_m)\big) > \tfrac12.
\]
Consequently,
\[
\|f_m-h_m\|_{L^1(\mathbb{R}^2)} \le \|f_m\|_{L^1(\mathbb{R}^2)} + \|h_m\|_{L^1(\mathbb{R}^2)} - b/2,
\]
which yields a contradiction with (4). This completes the proof.

\end{proof}

We remark that stronger versions of the above theorem can be derived by optimizing the proof.

\section*{Acknowledgements}
The author expresses his deepest thanks to Dirk Werner, who introduced him to the theory of Daugavet spaces, posed the question about a variant of Theorem~3.1, and provided valuable advice on several parts of this paper, improving readability and correcting flaws in some arguments.

\bibliographystyle{plain}
\bibliography{mybib}

@article{daugavet1963property,
 author = {Daugavet, I. K.},
 title = {A property of complete continuous operators in the space {{\(C\)}}},
 fjournal = {Uspekhi Matematicheskikh Nauk [N. S.]},
 journal = {Usp. Mat. Nauk},
 issn = {0042-1316},
 volume = {18},
 number = {5(113)},
 pages = {157--158},
 year = {1963},
 language = {Russian},
 zbMATH = {3224999},
 Zbl = {0138.38603}
}

@article{lozanovskii1966almost,
 author = {Lozanovskij, G. J.},
 title = {{\"U}ber {Fastintegraloperatoren} in {KB}-{R{\"a}umen}},
 fjournal = {Vestnik Leningradskogo Universiteta. Matematika, Mekhanika, Astronomiya},
 journal = {Vestn. Leningr. Univ., Mat. Mekh. Astron.},
 issn = {0024-0850},
 volume = {21},
 number = {2},
 pages = {35--44},
 year = {1966},
 language = {Russian},
 zbMATH = {3295192},
 Zbl = {0185.22901}
}

@article{kadets2000banach,
 author = {Kadets, Vladimir M. and Shvidkoy, Roman V. and Sirotkin, Gleb G. and Werner, Dirk},
 title = {Banach spaces with the {Daugavet} property},
 fjournal = {Comptes Rendus de l'Acad{\'e}mie des Sciences. S{\'e}rie I. Math{\'e}matique},
 journal = {C. R. Acad. Sci., Paris, S{\'e}r. I, Math.},
 issn = {0764-4442},
 volume = {325},
 number = {12},
 pages = {1291--1294},
 year = {1997},
 language = {French},
 doi = {10.1016/S0764-4442(97)82356-7},
 zbMATH = {1143356},
 Zbl = {0914.46022}
}

@article{abrahamsen2020delta,
 author = {Abrahamsen, T. A. and Haller, R. and Lima, V. and Pirk, K.},
 title = {Delta- and {Daugavet} points in {Banach} spaces},
 fjournal = {Proceedings of the Edinburgh Mathematical Society. Series II},
 journal = {Proc. Edinb. Math. Soc., II. Ser.},
 issn = {0013-0915},
 volume = {63},
 number = {2},
 pages = {475--496},
 year = {2020},
 language = {English},
 doi = {10.1017/S0013091519000567},
 zbMATH = {7191422},
 Zbl = {1445.46009}
}

@article{ivakhno2007daugavet,
 author = {Ivakhno, Yevgen and Kadets, Vladimir and Werner, Dirk},
 title = {The {Daugavet} property for spaces of {Lipschitz} functions},
 fjournal = {Mathematica Scandinavica},
 journal = {Math. Scand.},
 issn = {0025-5521},
 volume = {101},
 number = {2},
 pages = {261--279},
 year = {2007},
 language = {English},
 doi = {10.7146/math.scand.a-15044},
 zbMATH = {5361546},
 Zbl = {1177.46010}
}

@article{garcia2018characterisation,
 author = {Garc{\'{\i}}a-Lirola, Luis and Proch{\'a}zka, Anton{\'{\i}}n and Rueda Zoca, Abraham},
 title = {A characterisation of the {Daugavet} property in spaces of {Lipschitz} functions},
 fjournal = {Journal of Mathematical Analysis and Applications},
 journal = {J. Math. Anal. Appl.},
 issn = {0022-247X},
 volume = {464},
 number = {1},
 pages = {473--492},
 year = {2018},
 language = {English},
 doi = {10.1016/j.jmaa.2018.04.017},
 zbMATH = {6867358},
 Zbl = {1397.46008}
}

@book{lee2003smooth,
 author = {Lee, John M.},
 title = {Introduction to smooth manifolds},
 edition = {2nd revised ed},
 fseries = {Graduate Texts in Mathematics},
 series = {Grad. Texts Math.},
 issn = {0072-5285},
 volume = {218},
 isbn = {978-1-4419-9981-8; 978-1-4419-9982-5},
 year = {2013},
 publisher = {New York, NY: Springer},
 language = {English},
 doi = {10.1007/978-1-4419-9982-5},
 zbMATH = {6034615},
 Zbl = {1258.53002}
}

@article{thomassen1992jordan,
 author = {Thomassen, Carsten},
 title = {The {Jordan}-{Sch{\"o}nflies} theorem and the classification of surfaces},
 fjournal = {American Mathematical Monthly},
 journal = {Am. Math. Mon.},
 issn = {0002-9890},
 volume = {99},
 number = {2},
 pages = {116--130},
 year = {1992},
 language = {English},
 doi = {10.2307/2324180},
 zbMATH = {61217},
 Zbl = {0773.57001}
}

@article{bing1946generalizations,
 author = {Bing, R. H.},
 title = {Generalizations of two theorems of {Janiszewski}. {I}, {II}},
 fjournal = {Bulletin of the American Mathematical Society},
 journal = {Bull. Am. Math. Soc.},
 issn = {0002-9904},
 volume = {51},
 pages = {954--960},
 year = {1945},
 language = {English},
 doi = {10.1090/S0002-9904-1945-08480-8},
 zbMATH = {3098580},
 Zbl = {0060.40402}
}

@book{elstrodt1996mass,
 author = {Elstrodt, J{\"u}rgen},
 title = {Ma{{\ss}}- und {Integrationstheorie}},
 edition = {8th enlarged and updated edition},
 isbn = {978-3-662-57938-1; 978-3-662-57939-8},
 year = {2018},
 publisher = {Heidelberg: Springer Spektrum},
 language = {German},
 doi = {10.1007/978-3-662-57939-8},
 zbMATH = {6945708},
 Zbl = {1396.28001}
}

@book{evans2018measure,
 author = {Evans, Lawrence Craig and Gariepy, Ronald F.},
 title = {Measure theory and fine properties of functions},
 edition = {Revised ed.},
 fseries = {Textbooks in Mathematics},
 series = {Textb. Math.},
 isbn = {978-1-4822-4238-6; 978-1-4822-4240-9},
 year = {2015},
 publisher = {Boca Raton, FL: CRC Press},
 language = {English},
 zbMATH = {6413063},
 Zbl = {1310.28001}
}

@article{shvydkoy2000geometric,
 author = {Shvydkoy, R. V.},
 title = {Geometric aspects of the {Daugavet} property},
 fjournal = {Journal of Functional Analysis},
 journal = {J. Funct. Anal.},
 issn = {0022-1236},
 volume = {176},
 number = {2},
 pages = {198--212},
 year = {2000},
 language = {English},
 doi = {10.1006/jfan.2000.3626},
 zbMATH = {1529982},
 Zbl = {0964.46006}
}

\end{document}